\documentclass[10pt]{amsart}

\title{Fat flats in rank one manifolds}

\author{D. Constantine}
\address{
Department of Mathematics and Computer Science \\
Wesleyan University \\
Middletown, CT 06459}
\email{dconstantine@wesleyan.edu}

\author{J.-F. Lafont}
\address{Department of Mathematics\\
                 Ohio State University\\
                 Columbus, Ohio 43210}
\email{jlafont@math.ohio-state.edu}

\author{D. B. McReynolds}
\address{Department of Mathematics\\
                Purdue University\\
                 West Lafayette, IN 47907}
\email{dmcreyno@purdue.edu}

\author{D. J. Thompson}
\address{Department of Mathematics\\
                 Ohio State University\\
                 Columbus, Ohio 43210}
\email{thompson.2455@osu.edu}

\date{\today}

\thanks{D.C. thanks the Ohio State University Math Department for hosting him for a semester during which much of this work was completed. J.-F.L. is supported by NSF grant DMS-1510640. D.M. is supported by NSF grant DMS-1408458. D.T. is supported by NSF grant DMS-1461163.}

\newtheorem{thm}{Theorem}[section]

\newtheorem{thmx}{Theorem}
\newtheorem{corx}[thmx]{Corollary}

\newtheorem{lem}[thm]{Lemma}
\newtheorem{prop}[thm]{Proposition}
\newtheorem{cor}[thm]{Corollary}

\theoremstyle{definition}
\newtheorem*{rem}{Remark}

\newtheorem{defn}[thm]{Definition}

\numberwithin{equation}{section}

\def\Pb{\ifmmode{\Bbb P}\else{$\Bbb P$}\fi}
\def\Z{\ifmmode{\Bbb Z}\else{$\Bbb Z$}\fi}
\def\Q{\ifmmode{\Bbb Q}\else{$\Bbb Q$}\fi}
\def\C{\ifmmode{\Bbb C}\else{$\Bbb C$}\fi}
\def\R{\ifmmode{\Bbb R}\else{$\Bbb R$}\fi}
\def\H{\ifmmode{\Bbb H}\else{$\Bbb H\Bbb N$}\fi}

\def\N{\ifmmode{\Bbb N}\else{$\Bbb N$}\fi}

\newcommand{\set}[1]{\left\{#1\right\}}
\newcommand{\abs}[1]{\left\vert#1\right\vert}
\newcommand{\innp}[1]{\left< #1 \right>}
\newcommand{\pr}[1]{\left( #1 \right) }
\def \Sing{\mathrm{Sing}}

\usepackage{graphicx} 
\usepackage{amsmath,amssymb,amsthm,amsfonts,amscd,flafter,epsf,bm}
\usepackage{mathtools}
\usepackage{graphics}
\usepackage{MnSymbol}
\usepackage{epsfig}
\usepackage{psfrag}
\usepackage{color}
\usepackage{tikz}
\usetikzlibrary{patterns}

\input xy
\xyoption{all}

\begin{document}

\begin{abstract}

We study closed nonpositively curved Riemannian manifolds $M$ which admit `fat $k$-flats': that is, the universal cover $\tilde M$ contains a positive radius neighborhood of a $k$-flat on which the sectional curvatures are identically zero. We investigate how the {fat $k$-flats} affect the cardinality of the collection of closed geodesics. Our first main result is to construct rank $1$ nonpositively curved manifolds with a {fat $1$-flat} which corresponds to a \emph{twisted cylindrical neighborhood} of a geodesic on $M$. As a result, $M$ contains an embedded closed geodesic with a flat neighborhood, but $M$ nevertheless has only countably many closed geodesics. Such metrics can be constructed on finite covers of arbitrary odd-dimensional finite-volume hyperbolic manifolds. Our second main result is to prove a closing theorem for fat flats, which implies that a manifold $M$ with a {fat $k$-flat} contains an immersed, totally geodesic $k$-dimensional flat closed submanifold. This guarantees the existence of uncountably many closed geodesics when $k \geq 2$.  Finally, we collect results on thermodynamic formalism for the class of manifolds considered in this paper.
\end{abstract}

\maketitle

\setcounter{secnumdepth}{2}

\setcounter{section}{0}


\section{Introduction}

A basic characteristic of a dynamical system is the cardinality of its collection of periodic orbits.  A basic characteristic of a manifold is the cardinality of its collection of closed geodesics. These correspond to the periodic orbits for the geodesic flow of the manifold. It is interesting to see what cardinalities can be achieved for a given class of dynamical systems and to investigate conditions that may give restrictions on what is possible. For compact nonpositively curved manifolds, each free homotopy class of loops contains at least one closed geodesic, so the cardinality of the collection of closed geodesics is at least countably infinite. If a free homotopy class contains two geometrically distinct closed geodesics, then the flat strip theorem implies that the two closed geodesics span an isometrically embedded flat cylinder, giving uncountably many closed geodesics. However, while two distinct closed geodesics in the same homotopy class guarantee the existence of a flat strip, the question of whether the converse holds is much more subtle. We investigate this question in this paper for a large class of rank $1$ manifolds.

Recall that the \emph{rank} of a geodesic in a nonpositively curved Riemannian manifold $M$ is the dimension of the space of parallel Jacobi fields for the geodesic. The \emph{rank of $M$} is the minimum rank over all geodesics for $M$. The celebrated rank rigidity theorem of Ballmann \cite{ballmann-rank} and Burns and Spatzier \cite{burns-spatzier} completely characterizes closed nonpositively curved higher rank manifolds: they are either locally symmetric, or their universal cover splits isometrically  as a product. Thus, all other nonpositively curved Riemannian manifolds are rank $1$.

A {\it fat $k$-flat} in $\tilde M$ is a $k$-flat in $\tilde M$ with a neighborhood isometric to $\mathbb R^k \times B_w$, where $B_w \subset \mathbb R^{n-k}$ is a $w$-radius ball ($w>0$). In the particular case $k=1$, a fat $1$-flat is a geodesic $\tilde \gamma \subset \tilde M$  which has such a neighborhood. We call the  product neighborhood around $\tilde \gamma$ a \emph{flat cylindrical neighborhood} of $\tilde\gamma $. We occasionally abuse terminology and call a closed immersed submanifold in $M$ {\it fat} if it lifts to a fat flat in the universal cover $\tilde M$.   

Nonpositively curved manifolds where the zero sectional curvatures are concentrated on flats are the simplest class of rank $1$ manifolds for which the geodesic flow is not Anosov. This class includes metrics for which the geodesic flow has the weak specification property, for example, if the curvature vanishes only on a finite union of closed geodesics \cite[\S3.1]{CLT}. Manifolds for which the zero sectional curvatures are concentrated on \emph{fat flats} are perhaps the simplest class of rank $1$ manifolds for which the geodesic flow never has the weak specification property \cite[Theorem 3.4]{CLT}. This provides motivation to study this relatively simple class of nonpositively curved manifolds, whose geodesic flows exhibit dynamical behavior fundamentally different from the hyperbolic dynamics that occurs in negative curvature.

For a rank $1$ manifold, one might imagine that a {fat flat} is an obstruction to having {only} countably many closed geodesics in $M$. This is true in the surface case where Cao and Xavier have shown that flat strips close up to an immersion of  the product of $S^1$ with an interval \cite{cao-xavier}. We show that this phenomenon does not persist for odd-dimensional manifolds. Rank $1$ manifolds with flat cylindrical neighborhoods may admit only countably many closed geodesics. On the other hand, we will show that if $\tilde M$ has a {fat $k$-flat} for $k\geq 2$, then this forces the manifold to have an uncountable collection of closed geodesics. 

Our first main result is to construct examples  of rank $1$ manifolds in any odd dimension for which a countable collection of closed geodesics coexists with the presence of a flat cylindrical neighborhood. We call an open neighborhood of a closed geodesic $\gamma$ in $M$ a {\it twisted cylindrical neighborhood} if it lifts to a flat cylindrical neighborhood of $\tilde \gamma$ in $\tilde M$, and  the holonomy around the periodic geodesic $\gamma$, given by a matrix in $\mathrm{SO}(n-1)$, is non-trivial. We rigorously show that the twisting matrix can prevent any geodesic in the twisted cylinder neighborhood other than the central geodesic $\gamma$ from being closed. More precisely, we have the following: 

\begin{thmx} \label{thmA}
Let $M$ be an odd-dimensional, finite volume, hyperbolic $(2n+1)$-manifold. Then 
$M$ has a finite cover $\bar M \rightarrow M$  which supports nonpositively curved Riemannian metrics $g$ so that some closed geodesic $\gamma \subset (\bar M, g)$ has a twisted cylindrical neighborhood $N$ with holonomy having all eigenvalues of the form $e^{i\alpha}$ for $\alpha$ an irrational multiple of $\pi$. In particular:
\begin{enumerate}
	\item the sectional curvatures of $M$ are everywhere $\leq 0$,
	\item the closed geodesic $\gamma$ lifts to a fat 1-flat $\tilde \gamma \hookrightarrow \tilde M$, but 
	\item  there are only countably many closed geodesics in $(\bar M,g)$.
\end{enumerate}
\end{thmx}
The presence of only countably many closed geodesics implies that the flat neighborhood of $\gamma$ {\it cannot} be lifted to a (metrically) product neighborhood in any finite cover of $\bar M$, even though it lifts to such a product in the universal cover $\tilde M$. 

For the examples provided by Theorem \ref{thmA}, the twisted cylinder neighborhood contains uncountably many nonclosed geodesics. In contrast, for smooth rank $1$ \emph{surfaces},  no examples are known for which there exists a non-closed geodesic contained in a flat region for all time. Such geodesics have been ruled out for a large class of rank $1$ surfaces \cite{Wu13}, but the question of  whether such examples can exist in general remains open, and is related to the (difficult) question of ergodicity of the geodesic flow. See \cite{BM13} for an interesting discussion of these issues.  

The proof of Theorem \ref{thmA} falls into two parts. First, we argue that every odd-dimensional finite volume hyperbolic manifold contains a closed geodesic whose holonomy has no finite order eigenvalues. Although this result can be deduced from the work of Prasad--Rapinchuk \cite{PrasadRapinchuk}, we give a self-contained, elementary proof in Section \ref{sec:holonomy}. Once we have such a geodesic, we can pass to a finite cover so that the lift is an embedded geodesic $\gamma$ and gradually ``flatten out'' the metric in a neighborhood of $\gamma$ while keeping the same holonomy. This process is explained in 
\S \ref{sec:warping}. 

Our second main theorem is a closing theorem for {fat $k$-flats}.

\begin{thmx} \label{thmB}
Let $M$ be a closed, nonpositively curved Riemannian manifold, and 
suppose that $\tilde M^n$ contains a {fat $k$-flat}. Then $M^n$ contains
an immersed, totally geodesic, flat $k$-dimensional closed submanifold $N^k\hookrightarrow M^n$ such that
$\tilde N^k$ {is a fat $k$-flat}.
\end{thmx}

Note that, although the flat neighborhood of $\tilde N^k$ in the universal cover splits isometrically as a product metric, this might no longer be true at the level of the compact quotient $N^k$. There will be a holonomy representation encoding how the group $\pi_1(N^k)$ acts on the (trivial) normal bundle to $\tilde N^k$.

Theorem \ref{thmB} is established in Section \ref{sec:closing} and generalizes some unpublished work of Cao and Xavier \cite{cao-xavier}, who proved this result in the case $k=1, n=2$. The method of proof is to argue that if $N^k$ does not exist, then we can use a ``maximal'' fat flat $F$ to construct an increasing sequence of \emph{framed flat boxes} (see Definition \ref{flatbox}). We then use a compactness argument to take a limit of these flat boxes, yielding a {fat} $k$-flat strictly larger than $F$, and hence a contradiction. 

Our proof 
crucially uses the fact that the flat in $\tilde M$ is fat. Without this hypothesis, the question of closing flats for $C^\infty$ metrics remains a well-known open problem \cite{BM13}. This question is attributed to Eberlein. Affirmative answers have been obtained for codimension one flats \cite{S90} and when the metric is analytic \cite{BS91}. 

As a corollary of Theorem \ref{thmB}, we obtain the following result about the cardinality of the collection of closed geodesics for a manifold admitting a fat $k$-flat.

\begin{corx}
Let $M$ be a closed, nonpositively curved Riemannian manifold, and suppose that $\tilde M^n$ contains a fat $k$-flat. 
\begin{enumerate}
\item If $k=1$, then there exists a closed geodesic $\gamma$ in $M$ for which the neighborhood of the geodesic has zero sectional curvatures. By Theorem \ref{thmA} {it is possible that there are} no other closed geodesics contained in the flat neighborhood of $\gamma$.
\item If $k \geq 2$, then for all $1\leq l<k$, $M$ contains uncountably many immersed, closed and totally geodesic, flat $l$-submanifolds. In particular, there must be uncountably many closed geodesics.
\end{enumerate}
\end{corx}

In Section \ref{dynamics}, we add to our study of rank $1$ manifolds with fat flats by collecting results on thermodynamic formalism for the geodesic flow in this setting. This question has been studied recently by Burns, Climenhaga, Fisher, and Thompson \cite{BCFT}. We state the results proved there as they apply to rank $1$ manifolds for which the sectional curvatures are strictly negative away from zero curvature neighborhoods of some fat flats. 



\section{Holonomy of geodesics in hyperbolic manifolds}\label{sec:holonomy}

In this section, we prepare for the proof of Theorem \ref{thmA} by analyzing the possible holonomy along 
geodesics inside finite volume hyperbolic manifolds. Recall that every hyperbolic element $\gamma \in \mathrm{SO}_0(n,1)$ 
is conjugate to a matrix of the form $\begin{pmatrix} T_\gamma & \textrm{0}_{n-1,2} \\ \textrm{0}_{2,n-1} & D_\gamma \end{pmatrix}$ 
where $T_\gamma \in \mathrm{O}(n-1)$, $\textrm{0}_{i,j}$ denotes the $i,j$ matrix with zero entries, and 
$D_\gamma = \mathrm{diag}(\lambda_\gamma,\lambda_\gamma^{-1}) =  \begin{pmatrix} \lambda_\gamma & 0 \\ 0 & \lambda_\gamma^{-1} \end{pmatrix}$ with $\lambda_\gamma \in \R$ and $\abs{\lambda_\gamma} > 1$ (see \cite[Sec 5.1]{LongReid}). 

The hyperbolic element $\gamma$ acts on $\mathbb H^n$ by translation along a geodesic axis. The matrix $D_\gamma$ encodes the 
translational distance along the axis, while the matrix $T_\gamma$ captures the rotational effect around the axis. The matrix 
$T_\gamma$ will be called the {\it holonomy} of the element $\gamma$. The element is said
to be {\it purely irrational} if every eigenvalue of $T_\gamma$ has infinite (multiplicative) order. At the other extreme, the element is said to be {\it purely
translational} if $T_\gamma$ is the identity matrix.

In Section \ref{subsec:key-prop}, we prove the
key proposition, which establishes the existence of closed geodesics with the maximal number of
infinite-order eigenvalues in the holonomy. In Section \ref{subsec:consequences}, we explain how,
at the other extreme, our method can also be used to produce geodesics that are purely hyperbolic
(i.e., whose holonomy map is the identity).

%

\subsection{Geodesics with purely irrational holonomy}\label{subsec:key-prop}
In this section, we will focus on finding hyperbolic elements whose holonomy have the maximal number of eigenvalues with infinite order.
More precisely, we show:

\begin{prop}\label{P:HolonomyProp}
Let $\Gamma < \mathrm{SO}_0(n,1)$ be a lattice with $n \geq 3$. 
\begin{itemize}
	\item[(a)] If $n$ is odd, then there exists a purely irrational element $\gamma \in \Gamma$.
	\item[(b)] If $n$ is even, then there exists a hyperbolic element $\gamma \in \Gamma$ such that $n-2$ of the eigenvalues of $T_\gamma$ have infinite order.
\end{itemize}
\end{prop}

\begin{rem}
Before proceeding, we make a few remarks.
\begin{itemize}
	\item[(i)] When $n$ is even, it is well known that $T_\gamma^2$ must have $1$ as an eigenvalue. In particular, in (b) the hyperbolic element $\gamma \in \Gamma$ has holonomy with the maximum number of eigenvalues with infinite order.
	\item[(ii)] Proposition \ref{P:HolonomyProp} improves on Long and Reid \cite[Thm 1.2]{LongReid}, who proved that there exists a hyperbolic element $\gamma \in \Gamma$ such that $T_\gamma$ has infinite order. 
	\item[(iii)] The proof of Proposition \ref{P:HolonomyProp} extends to any finitely generated Zariski dense subgroup $\Gamma$ of $\mathrm{SO}_0(n,1)$ with coefficients in the algebraic closure $\overline{\Q}$ of $\Q$. 
\end{itemize}
\end{rem}

As noted in \cite{LongReid}, both \cite[Thm 1.2]{LongReid} and Proposition \ref{P:HolonomyProp} can be deduced from 
\cite[Thm 1]{PrasadRapinchuk}. We will instead give an elementary proof of this result. 

\vskip 10pt

To prove Proposition \ref{P:HolonomyProp}, we begin with some background material on nondegenerate bilinear forms over finite fields and their associated isometry groups. We refer the reader to \cite{Omeara} (see also \cite[Sec 3.7]{Wilson}). Throughout, $\mathbb{F}_q$ will denote the unique finite field of cardinality $q=p^t$ where $p \in \N$ is an odd prime and $t \in \N$. There are two nondegenerate bilinear forms $B$ on $\mathbb{F}_q^2$ up to isometry. The isotropic form $B_h$ is given in coordinates by the matrix $\mathrm{diag}(1,1)$. The anisotropic form $B_a$ is given in coordinates by $\mathrm{diag}(1,\alpha)$ where $\alpha \in \mathbb{F}_q^\times - (\mathbb{F}_q)^2$. Alternatively, $B_a$ is the bilinear form associated to the quadratic form given by the norm $N_{\mathbb{F}_{q^2}/\mathbb{F}_q}$, where $\mathbb{F}_{q^2}$ is the unique quadratic extension of $\mathbb{F}_q$, and $N_{\mathbb{F}_{q^2}/\mathbb{F}_q}(\alpha) = \alpha \overline{\alpha}$, where $\overline{\alpha}$ is the Galois conjugate of $\alpha$. For a nondegenerate bilinear form $B$ on $\mathbb{F}_q$, we denote the associated orthogonal and special orthogonal groups by $\mathrm{O}(B,\mathbb{F}_q)$ and $\mathrm{SO}(B,\mathbb{F}_q)$. The group $\mathrm{SO}(B_h,\mathbb{F}_q)$ is a cyclic group of order $q-1$. If $\lambda_q$ is a generator for $\mathbb{F}_q^\times$, then the generator for $\mathrm{SO}(B_h,\mathbb{F}_q)$ can be taken to be conjugate in $\mathrm{GL}(2,\mathbb{F}_q)$ to $\mathrm{diag}(\lambda_q,\lambda_q^{-1})$. The group $\mathrm{SO}(B_a,\mathbb{F}_q)$ is a cyclic group of order $q+1$. If $\lambda_{q,2} \in \mathbb{F}_{q^2}$ is a generator of the cyclic group of elements of norm 1, then we can take a generator of $\mathrm{SO}(B_a,\mathbb{F}_q)$ to be conjugate in $\mathrm{GL}(2,\mathbb{F}_{q^2})$ to $\mathrm{diag}(\lambda_{q,2},\lambda_{q,2}^{-1})$. There are two equivalence classes of bilinear forms on $\mathbb{F}_q^{2n+1}$ given by $B = B_h \oplus \dots \oplus B_h \oplus \innp{\alpha}$ where $\innp{\alpha}$ is a 1-dimension bilinear form given by $\alpha \in \mathbb{F}_q^\times$; the equivalence class is determined by whether $\alpha \in (\mathbb{F}_q)^2$. However, the associated orthogonal and special orthogonal groups are isomorphic. We denote the orthogonal, special orthogonal groups in this case by $\mathrm{O}(2n+1,q)$ and $\mathrm{SO}(2n+1,q)$. Note that $\prod_{i=1}^n \mathrm{SO}(B_h,\mathbb{F}_q) < \mathrm{SO}(2n+1,q)$. On $\mathbb{F}_q^{2n}$, there are two isomorphism types for both $\mathrm{O}(2n,q)$ and $\mathrm{SO}(2n,q)$. The first isomorphism type is associated to the bilinear form $B_{+,n} = B_h \oplus \dots \oplus B_h$ and we denote the associated orthogonal, special orthogonal groups by $\mathrm{O}_+(2n,q)$ and $\mathrm{SO}_+(2n,q)$. In this case, we have $\prod_{i=1}^n \mathrm{SO}(B_h,\mathbb{F}_q) < \mathrm{SO}_+(2n,q)$. The second isomorphism type is associated to the bilinear form $B_{-,n} = B_h \oplus \dots \oplus B_h\oplus B_a$ and we denote the associated orthogonal and special orthogonal groups by $\mathrm{O}_-(2n,q)$ and $\mathrm{SO}_-(2n,q)$. In this case, we have $\pr{\prod_{i=1}^{n-1} \mathrm{SO}(B_h,\mathbb{F}_q)} \times \mathrm{SO}(B_a,\mathbb{F}_q) < \mathrm{SO}_-(2n,q)$.  Thus we obtain the following lemma (see \cite[Sec 3.7.4]{Wilson}).

\begin{lem}\label{L:PowerLemma}
Let $n \in \N$ with $n \geq 1$ and $q = p^t$ for an odd prime $p$.
\begin{itemize}
	\item[(a)] $\mathrm{O}(2n+1,q)$ and $\mathrm{SO}(2n+1,q)$ have an element with $n$ eigenvalues equal to $\lambda_q$, $n$ eigenvalues equal to $\lambda_q^{-1}$, and one eigenvalue equal to $1$.
	\item[(b)] $\mathrm{O}_+(2n,q)$ and $\mathrm{SO}_+(2n,q)$ have an element with $n$ eigenvalues equal to $\lambda_q$ and $n$ eigenvalues equal to $\lambda_q^{-1}$.
	\item[(c)] $\mathrm{O}_-(2n,q)$ and $\mathrm{SO}_-(2n,q)$ have an element with $n-1$ eigenvalues equal to $\lambda_q$, $n-1$ eigenvalues equal to $\lambda_q^{-1}$, and one eigenvalue each equal to $\lambda_{q,2},\lambda_{q,2}^{-1}$.
\end{itemize}
\end{lem}

Given a lattice $\Gamma < \mathrm{SO}_0(n,1)$ with $n\geq 3$, we can conjugate $\Gamma$ so that the field of definition $k_\Gamma$ is real number field (see \cite[Sec 4.1]{LongReid}).  If $\gamma \in\Gamma$ has an eigenvalue of finite multiplicative order, then the splitting field for the characteristic polynomial will contain a root of unity $\zeta_m$ for some $m$. By \cite[Prop 2.1]{LongReid}, there exists $M_\Gamma \in \N$ depending only on $n$ and $[k_\Gamma:\Q]$ such that $m \leq M_\Gamma$. Thus, we have the following:

\begin{lem}\label{L:PowerBound}
If $\gamma \in \Gamma$ and $\lambda$ is an eigenvalue for $\gamma$, then either $\lambda^m = 1$ for some $m \leq M_\Gamma$ or $\lambda$ has infinite order.
\end{lem}

\begin{proof}[Proof of Proposition \ref{P:HolonomyProp}]
Setting $k_\Gamma$ to be the field of definition for $\Gamma$ and $\mathcal{O}_{k_\Gamma}$ to be the ring of $k_\Gamma$-integers, let $R_\Gamma$ be ring generated over $\mathcal{O}_{k_\Gamma}$ by the matrix coefficients $\Gamma < \mathrm{SO}_0(n,1;k_\Gamma)$. The ring $R_\Gamma = \mathcal{O}_{k_\Gamma}[\mathfrak{p}_1^{-1},\dots,\mathfrak{p}_r^{-1}]$ where $\mathfrak{p}_j < \mathcal{O}_{k_\Gamma}$ are prime ideals. For every prime ideal $\mathfrak{p} < \mathcal{O}_{k_\Gamma}$ with $\mathfrak{p} \ne \mathfrak{p}_1,\dots,\mathfrak{p}_r$, the associated prime ideal $\mathfrak{P} = \mathfrak{p}R_\Gamma < R_\Gamma$ satisfies $R_\Gamma/\mathfrak{P} \cong \mathcal{O}_{k_\Gamma}/\mathfrak{p} \cong \mathbb{F}_q$, where $q=p^t$ for some $t \leq [k_\Gamma:\Q]$. When $n$ is even, we have homomorphisms $r_\mathfrak{P}\colon \mathrm{SO}_0(n,1;R_\Gamma) \to \mathrm{SO}(n+1,q)$ given by reducing the coefficients modulo $\mathfrak{P}$. By strong approximation (see \cite[Thm 5.3 (1)]{LongReid}), there is a cofinite set of prime ideals $\mathfrak{P}$ of $R_\Gamma$ such that $\Omega(n+1,q) \leq r_\mathfrak{P}(\Gamma) < \mathrm{SO}(n+1,q)$, where $\Omega(n+1,q)$ is the commutator subgroup of $\mathrm{SO}(n+1,q)$ and has index two in $\mathrm{SO}(n+1,q)$. When $n$ is odd, all of the above carries over with $\overline{r}_\mathfrak{P}\colon \mathrm{PSO}_0(n,1;R_\Gamma) \to \mathrm{PSO}_{\pm}(n+1,q)$ in place of $r_\mathfrak{P}$. By strong approximation (see \cite[Thm 5.3 (2)]{LongReid}), there is an infinite set of prime ideals $\mathfrak{P}$ of $R_\Gamma$ such that $\Omega_{\pm}(n+1,q) \leq \overline{r}_\mathfrak{P}(\Gamma) < \mathrm{PSO}_{\pm}(n+1,q)$ where $\Omega_\pm(n+1,q)$ is the commutator subgroup of $\mathrm{PSO}_\pm(n+1,q)$ and has index two in $\mathrm{PSO}(n+1,q)$. Additionally, the group $\mathrm{PSO}(n+1,q)$ is the quotient of $\mathrm{SO}(n+1,q)$ by its center which has order $\mathrm{gcd}(4,q-1)$. We now take a prime ideal $\mathfrak{P} < R_\Gamma$ such that $\abs{R_\Gamma/\mathfrak{P}} - 1 = q-1 > 8M$. By Lemma \ref{L:PowerLemma} and the above discussion, there exists $g \in r_\mathfrak{P}(\Gamma), \overline{r}_\mathfrak{P}(\Gamma)$, respectively, such that the following holds. When $n$ is odd, every eigenvalue of $g$ has multiplicative order at least $(q-1)/8$, and when $n$ is even, all but one eigenvalue of $g$ has multiplicative order at least $(q-1)/2$. For any $\gamma \in r_\mathfrak{P}^{-1}(g),\overline{r}_\mathfrak{P}^{-1}(g)$, respectively, we see that either all of the eigenvalues of $\gamma$ have multiplicative order at least $(q-1)/8$ or all but one of the eigenvalues of $\gamma$ have multiplicative order at least $(q-1)/2$; note that the characteristic polynomial $c_\gamma(t) \in R_\Gamma[t]$ of $\gamma$ and the characteristic polynomial $c_g(t) \in \mathbb{F}_q[t]$ of $g$ are related via $c_\gamma(t) = c_g(t) \mod \mathfrak{P}$. By selection of $\mathfrak{P}$ we see that either all of the eigenvalues of $\gamma$ have multiplicative order greater than $M_\Gamma$ or all but one of the the eigenvalues have multiplicative order greater than $M_\Gamma$ depending only on the odd/even parity of $n$. Lemma \ref{L:PowerBound} completes the proof.
\end{proof}

\begin{rem}
When $n$ is odd, so long as $n \ne 7$ or $\Gamma$ is not an arithmetic lattice arising from triality, the infinite set of primes in the proof of Proposition \ref{P:HolonomyProp} can be taken to be a confinite set. When $n=7$ and $\Gamma$ arises from triality, this set can only be taken to be infinite as $\Gamma$ will have infinite many primes with image contained in $\mathrm{G}_2$.
\end{rem}

\subsection{Purely hyperbolic geodesics}\label{subsec:consequences}
Proposition \ref{P:HolonomyProp} provides us with elements whose holonomy has the maximal number of infinite-order eigenvalues.
It is reasonable to ask if one can instead find elements with some prescribed (fewer) number of infinite-order eigenvalues. For example,
if there are no infinite-order eigenvalues, then one is looking for an element whose finite power is purely hyperbolic. Using our Proposition
\ref{P:HolonomyProp}, we obtain the following:

\begin{cor}\label{C:ArithmeticHolonomy}
Let $n,j$ be integers such that $n \geq 3$, $j$ is even, and $0 \leq j \leq n-1$.
\begin{itemize}
\item[(a)]
If $n$ is odd and $j>0$, then for any arithmetic lattice $\Gamma < \mathrm{SO}_0(n,1)$ that does not arise from triality (e.g, if $n \ne 7$), there exists a hyperbolic element $\gamma \in \Gamma$ such that $T_\gamma$ has $j$ eigenvalues with infinite order and has $1$ as an eigenvalue with multiplicity $(n-1)-j$.
\item[(b)]
If $n$ is even, then for any arithmetic lattice $\Gamma < \mathrm{SO}_0(n,1)$, there exists a hyperbolic element $\gamma \in \Gamma$ such that $T_\gamma$ has $j$ eigenvalues with infinite order and has $1$ as an eigenvalue with multiplicity $(n-1)-j$. In particular, $\Gamma$ has a purely hyperbolic element.
\item[(c)]
If $n$ is odd, there exist infinitely many commensurability classes of arithmetic lattices $\Gamma < \mathrm{SO}_0(n,1)$ that have a purely hyperbolic element.
\end{itemize}
\end{cor}

\begin{rem}
Before proving Corollary \ref{C:ArithmeticHolonomy}, we make a few more remarks.
\begin{itemize}
\item[(i)]
Having a hyperbolic element that satisfies any of the properties in Proposition \ref{P:HolonomyProp} or Corollary \ref{C:ArithmeticHolonomy} is a commensurability invariant since these properties are stable under conjugation and finite powers. Hence, we can take $\Gamma < G$ for any $G$ isogenous to $\mathrm{SO}_0(n,1)$ (e.g., $G=\mathrm{Isom}(\H^n)$).
\item[(ii)]
There are examples of arithmetic hyperbolic 3--manifolds without any purely hyperbolic elements. Viewing the group of orientation preserving isometries of hyperbolic 3--space as $\mathrm{PSL}(2,\C)$, a purely hyperbolic element $\gamma \in \mathrm{PSL}(2,\C)$ must have a real trace. Chinburg and Reid \cite{ChinburgReid} constructed infinitely many commensurability classes of arithmetic hyperbolic 3-manifolds for which every nontrivial element has a trace in $\C - \R$. Consequently, one cannot improve (c) in the case of $n=3$. 

\end{itemize}
\end{rem}

In the proof of Corollary \ref{C:ArithmeticHolonomy}, we require the following consequence of the classification of arithmetic lattices in $\mathrm{SO}_0(n,1)$. For a more detailed discussion, we refer the reader to \cite[Sec 6.4]{Witte} (see also \cite{MeyerThesis},  \cite{Meyer}).

\begin{lem}\label{L:Subgroups}
\begin{itemize}
	\item[(a)] If $n$ is odd, then for every arithmetic (cocompact) lattice $\Gamma < \mathrm{SO}_0(n,1)$ that does not arise from triality (e.g, if $n \ne 7$) and every integer $3 \leq 2j+1\leq n$, there exists $\mathbf{G}<\mathrm{SO}_0(n,1)$ with $\mathbf{G} \cong \mathrm{SO}_0(2j+1,1)$ such that $\Delta_j = \mathbf{G} \cap \Gamma$ is a (cocompact) lattice in $\mathbf{G}$.
	\item[(b)] If $n$ is even, then for every arithmetic (cocompact) lattice $\Gamma < \mathrm{SO}_0(n,1)$ and every integer $2 \leq j < n$, there exists $\mathbf{G}<\mathrm{SO}_0(n,1)$ with $\mathbf{G} \cong \mathrm{SO}_0(j,1)$ such that $\Delta_j = \mathbf{G} \cap \Gamma$ is a (cocompact) lattice in $\mathbf{G}$.
	\item[(c)] If $n$ is odd, then there exists infinitely many commensurability classes of arithmetic (cocompact) lattices $\Gamma$ in $\mathrm{SO}_0(n,1)$ with $\mathbf{G}<\mathrm{SO}_0(n,1)$ and $\mathbf{G} \cong \mathrm{SO}_0(2,1)$ such that $\Delta_2 = \mathbf{G} \cap \Gamma$ is a (cocompact) lattice in $\mathbf{G}$.
\end{itemize}
\end{lem}

\begin{proof}[Proof of Corollary \ref{C:ArithmeticHolonomy}]
For (a), we apply Lemma \ref{L:Subgroups}(a) for $j+1$, and deduce, for every arithmetic lattice $\Gamma < \mathrm{SO}_0(n,1)$ that does not arise from triality, that there exists $\mathbf{G}<\mathrm{SO}_0(n,1)$ with $\mathbf{G} \cong \mathrm{SO}_0(j+1,1)$ such that $\Delta_j = \mathbf{G} \cap \Gamma$ is a lattice in $\mathbf{G}$. In particular, $\Delta_j$ is conjugate into $\mathrm{SO}_0(j+1,1)<\mathrm{SO}_0(n,1)$ where $\mathrm{SO}_0(j+1,1)$ corresponds to the subgroup 
\[ \set{\begin{pmatrix} \mathrm{I}_{n-j} & \textrm{0}_{n-j,j+1} \\ \textrm{0}_{j+1,n-j} & A \end{pmatrix} ~:~ A \in \mathrm{SO}_0(j+1,1)}. \]
Since $j+1$ is odd, by Proposition \ref{P:HolonomyProp}, we can find $\gamma \in \Delta_j$ such that $T_\gamma$ has $j$ eigenvalues of with infinite order. By construction, the remaining $(n-1)-j$ eigenvalues of $T_\gamma$ are 1. For (b), we apply Lemma \ref{L:Subgroups}(b) for $j+1$, and so for every arithmetic lattice $\Gamma < \mathrm{SO}_0(n,1)$, there exists $\mathbf{G}<\mathrm{SO}_0(n,1)$ with $\mathbf{G} \cong \mathrm{SO}_0(j+1,1)$ such that $\Delta_j = \mathbf{G} \cap \Gamma$ is a lattice in $\mathbf{G}$. The remainder of the proof is identical to part (a). To obtain a purely hyperbolic element, we apply Lemma \ref{L:Subgroups} for $j=0$ to obtain $\Delta_2 < \Gamma$ given by $\Delta_2 = \mathbf{G} \cap \Gamma$ with $\mathbf{G} \cong \mathrm{SO}_0(2,1)$, $\mathbf{G} < \mathrm{SO}_0(n,1)$. As any hyperbolic element $\gamma \in \Delta_2$ is purely hyperbolic (after taking the square if necessary), $\Gamma$ contains a purely hyperbolic element. For (c), by Lemma \ref{L:Subgroups}(c), there exist infinitely many commensurability classes of arithmetic lattices $\Gamma < \mathrm{SO}_0(n,1)$ with $\mathbf{G}<\mathrm{SO}_0(n,1)$ and $\mathbf{G} \cong \mathrm{SO}_0(2,1)$ such that $\Delta_2 = \mathbf{G} \cap \Gamma$ is a lattice in $\mathbf{G}$. As in (b), we conclude that $\Gamma$ contains a purely hyperbolic element.
\end{proof}


\section{Flat cylinders with purely irrational holonomy}\label{sec:warping}

In this section, we complete the proof of Theorem \ref{thmA}. The arguments here are differential geometric in nature: we ``flatten out'' the metric on a hyperbolic manifold inside a neighborhood of a suitably chosen geodesic. In order to do this, we start in Section \ref{subsec:interpol} by constructing smooth functions with certain specific properties. These functions are then used in Section \ref{subsec:warping} to radially interpolate from the hyperbolic metric (away from the geodesic) to a flat metric (near the geodesic). Finally, in Section \ref{subsec:thmB-proof}, we put together all the pieces and establish Theorem \ref{thmA}.

%

\subsection{Some smooth interpolating functions.}\label{subsec:interpol}

In this section, we establish the existence of smooth interpolating functions satisfying certain technical 
conditions. Specifically, we show the following:

\begin{prop}\label{interpolating-functions}
For $R$ sufficiently large, there exist $C^\infty$ functions $\sigma$ and $\tau$ satisfying $\sigma, \tau\geq 0$, $\sigma'\geq 1$, $\tau'\geq 0$ and $\sigma'', \tau'' \geq 0$ for all $r\geq 0$, and
	\[\sigma (r) = \begin{cases}
		\sinh(r), & r\geq R \\
		r, & r\leq 1/R
		\end{cases} 
	\hskip 0.5in 
	\tau (r) =\begin{cases}
		\cosh(r), & r\geq R \\
		1, & r\leq 1/R.
\end{cases} \]
\end{prop}

\begin{proof}

We start by taking the function:
\[	
	f_k(r) = 
		\begin{cases}			
			e^{\frac{-k^2}{k^2-x^2}}, & |x|\leq k \\
			
			0, & |x|> k.
		\end{cases}
\]
and form the functions $F_k(x) = \int_{-k}^x f_k(s)ds$ and $\rho(r) = F_k(r-(k+\frac{1}{k}))/F_k(k)$ for a value of $k$ to be chosen.
Finally, we form the functions $\sigma(r) = \rho(r)\sinh(r) + (1-\rho(r))r$ and $\tau(r) = \rho(r)\cosh(r) + (1-\rho(r))$.

It is straightforward to check that these functions $\sigma$ and $\tau$ have the correct large scale and small scale behavior 
(with $R=2k+1$), and it is also easy to check that $\sigma \geq 0$, $\tau \geq 0$, $\sigma ' \geq 1$, and $\tau ' \geq 0$.
The only subtlety lies in verifying the sign of $\sigma''$ and $\tau''$. For both of these, we can differentiate and see that the 
second derivative consists of a nonnegative term plus $\rho''(r) + \rho(r)$ times another nonnegative term. So establishing convexity
of both functions reduces to checking the inequality $\rho''(r) + \rho(r) \geq 0$ for all $r$, which we can easily verify holds 
whenever the parameter $k$ is at least $18$. These routine computations are left to the reader.
\end{proof}

%

\subsection{Some curvature estimates.}\label{subsec:curvature}

We provide a walk-through of some curvature computations that will be needed in
the proof of Theorem \ref{thmA}. We use the coordinate system $\{r, \theta, \phi, z\}$ on $\R^4 = \R^3 \times \R$
where $\{r, \theta, \phi\}$ are standard spherical coordinates on the $\R^3$ factor. In terms of this coordinate system,
consider the metric given by
\begin{equation}
h:=dr^2 + \sigma ^2(r) d\theta ^2 + \sigma ^2(r) \sin ^2(\theta) d\phi ^2 + \tau ^2(r) dz ^2,
\end{equation}
where $\sigma$ and $\tau$ are the functions constructed in Proposition \ref{interpolating-functions}. Note that $d\theta ^2 + \sin ^2(\theta) d\phi ^2$ is the standard round metric on $S^2$. This doubly warped product
metric will appear in the proof of Proposition \ref{flattening-out} in our next section. 

\begin{prop}\label{4d-NPC}
The metric $h$ on $\R^4$ has nonpositive sectional curvature.
\end{prop}

\begin{proof}
At every point not on the $z$-axis, an ordered basis for the tangent space is given by $\{ \frac{\partial}{\partial r}, \frac{\partial}{\partial \theta}, 
\frac{\partial}{\partial \phi}, \frac{\partial}{\partial z} \}$. Abusing notation, we will use the indices $r, \theta, \phi, z$ to denote the 
corresponding vectors in the ordered basis.

Some straightforward computations then show that the only nonzero components of the curvature $4$-tensor
$(u,v,w,z)=\langle R(u,v)w, z\rangle$ are given by:
\begin{align}\label{curvature1}
	R_{\theta r \theta r} &= - \sigma {\sigma''} &R_{\phi r \phi r} &= - \sigma \sigma '' \sin ^2 (\theta)\\
	R_{zr zr} & = - \tau {\tau ''} & R_{\phi \theta \phi \theta} &= \left(1 - (\sigma ')^2 \right)\sigma ^2 \sin ^2(\theta)\\
	\label{curvature3}R_{\theta z \theta z} & = - \sigma \sigma ' \tau \tau ' & R_{\phi z \phi z} &= - \sigma \sigma ' \tau \tau ' 			\sin^2(\theta)
\end{align}
For reference, these computations were included in a preprint version of this paper\footnote{Version 1 of arXiv:1704.00857 }. Similar computations appear in \cite[\S4.2.4]{pP16}
for doubly warped product metrics of spherical metrics, which differ slightly from the metric here, where one of the warping directions is flat.
From the properties of the functions $\sigma$ and $\tau$ given by Proposition \ref{interpolating-functions} it is now immediate that all 
six of these tensor components are nonpositive. 

Finally, we recall that, for an arbitrary tangent $2$-plane $H$, the sectional curvature $K(H)$ is computed by choosing any two linearly
independent vectors $u,v$, and calculating 
$$K(H):=\frac{(u,v,u,v)}{||u||\cdot ||v|| - \langle u, v \rangle}$$
(this expression is independent of the choice of vectors). In our setting, if $H$ is an arbitrary tangent $2$-plane, we can choose
orthonormal $u, v\in H$, and express them as a linear combination of the coordinate vectors $u= u_r \frac{\partial}{\partial r} +
u_\theta \frac{\partial}{\partial \theta} + u_\phi \frac{\partial}{\partial \phi} + u_z \frac{\partial}{\partial z}$ and $v= v_r \frac{\partial}{\partial r} +
v_\theta \frac{\partial}{\partial \theta} + v_\phi \frac{\partial}{\partial \phi} + v_z \frac{\partial}{\partial z}$. Then the sectional curvature 
of $K(H)$ computes to:
\begin{align*}
K(H) &= (u,v,u,v) = \langle R(u,v)u,v\rangle \\
&= (u_r^2v_\theta^2 + u_\theta^2 v_r^2)R_{\theta r\theta r} + (u_z^2v_r^2 + u_r^2v_z^2)R_{zrzr} + (u_z^2v_\theta^2 + u_\theta^2 v_z^2)R_{\theta z\theta z} \\
& +  (u_r^2v_\phi^2 + u_\phi^2 v_r^2)R_{\phi r\phi r} +  (u_\theta^2v_\phi^2 + u_\phi^2 v_\theta^2)R_{\phi \theta \phi \theta} + (u_z^2v_\phi^2 + u_\phi^2 v_z^2)R_{\phi z\phi z},
\end{align*}
which is a positive linear combination of six nonpositive expressions: see equations (\ref{curvature1})--(\ref{curvature3}). 
Thus we conclude that $K(H) \leq 0$. 
Since this holds for arbitrary tangent $2$-planes, this metric is nonpositively curved away from the $z$-axis, and a continuity argument
then shows it is nonpositively curved everywhere. This completes the proof of the proposition.
\end{proof}

As a corollary of the previous computation, we immediately obtain the following three-dimensional estimate. Consider $\mathbb R^3$
with cylindrical coordinates $\{r, \theta , z\}$, and with Riemannian metric
\begin{equation}
g:=dr^2 + \sigma ^2(r) d\theta ^2 + \tau ^2(r) dz ^2.
\end{equation}
This special metric on $\mathbb R^3$ will be used in the proof of Proposition \ref{flattening-out} (it is the $n=1$ case
in equation (\ref{warping})).

Recall that the fixed point set of any isometry is totally geodesic. Since $(\R ^3, g)$ is isometric to the fixed point set
of the isometric involution $(r, \theta, \phi, z) \mapsto (r, \theta , \pi - \phi, z)$ defined on $(\R^4, h)$, we immediately 
obtain the following:

\begin{cor}\label{3d-NPC}
The metric $g$ on $\R ^3$ has nonpositive sectional curvatures.
\end{cor}

%

\subsection{From twisted hyperbolic to twisted flat neighborhoods.}\label{subsec:warping}

Let $\rho \in \mathrm{SO}(2n)$ be a purely irrational rotation (i.e., one whose eigenvalues all have infinite order), and let $\ell>0$ be a positive real number. Form the space $C_{\ell, \rho}$ as follows. First consider the quadratic equation
	$$f(t)=t^2 - 2\cosh(\ell/2) t +1,$$
and let $\lambda >1$ be the larger of the two real roots of this polynomial. Form the diagonal matrix $D(\ell) = \mathrm{diag}(\lambda,\lambda) =  \begin{pmatrix} \lambda & 0 \\ 0 & \lambda^{-1} \end{pmatrix}$ and consider the element $g:= \begin{pmatrix} \rho & \textrm{0}_{n-1,2} \\ \textrm{0}_{2,n-1} & D(\ell) \end{pmatrix}$ where again $\textrm{0}_{i,j}$ denotes the $i \times j$ matrix with zero entries. This is an element in $\mathrm{SO}_0(2n+1,1) = \text{Isom}(\mathbb H^{2n+1})$ and so acts on hyperbolic space $\mathbb H^{2n+1}$. Finally, we define the space $C_{\ell, \rho}$ to be the quotient of $\mathbb H^{2n+1}$ by the $\mathbb Z$-action generated by the isometry $g$. The hyperbolic isometry $g$ leaves invariant a unique geodesic $\tilde \gamma$,
whose image in $C^{2n+1}_{\ell, \rho} = \mathbb H^{2n+1}/\langle g\rangle$ is the unique closed embedded geodesic $\gamma \hookrightarrow C^{2n+1}_{\ell, \rho}$. By construction, the length of $\gamma$ is $\ell$, and the holonomy along $\gamma$ is given by the matrix $\rho \in \mathrm{SO}(2n)$ (with respect to an 
appropriate choice of basis). The following result is immediate.

\begin{lem}\label{standard-nbhd}
Let $M$ be an arbitrary hyperbolic $(2n+1)$-manifold, and let $\eta \hookrightarrow M$ be an embedded closed geodesic with length $l$, holonomy $\rho \in \mathrm{SO}(2n)$, and normal injectivity radius $R>0$. Then the $R$-neighborhood of $\eta \hookrightarrow M$ is isometric to the $R$-neighborhood of $\gamma \hookrightarrow C^{2n+1}_{l,\rho}$.
\end{lem}

We now explain how to ``flatten out'' the hyperbolic metric on $C_{\ell, \rho}$ near the periodic geodesic $\gamma$, while retaining the same holonomy. This is the content of the following:

\begin{prop}\label{flattening-out}
If $R>{37}$, and $n\geq 1$, then $C^{2n+1}_{\ell, \rho}$ supports a Riemannian metric $g$ with the following properties:
\begin{enumerate}
\item outside of the $R$-neighborhood of $\gamma$, $g$ coincides with the hyperbolic metric,
\item inside the $(1/R)$-neighborhood of $\gamma$, $g$ is flat, 
\item the sectional curvature on every $2$-plane is nonpositive, and
\item the holonomy around $\gamma$ is given by the matrix $\rho$.
\end{enumerate}
\end{prop}

\begin{proof}
Consider $\R^{2n+1}$ with generalized cylindrical coordinates and equipped with the metric 
	\begin{equation}\label{eqn:hyp-metric}
		dr^2 + \sinh^2(r) d\theta_{S^{2n-1}} ^2  + \cosh^2(r) dz ^2
	\end{equation}
where $d\theta_{S^{2n-1}} ^2$ denotes the standard (round) metric on the unit sphere $S^{2n-1}\subset \mathbb R^{2n}$. This space is isometric to $\mathbb H^{2n+1}$, and without loss of generality, we may assume that the isometry identifies $\tilde \gamma$ with the $z$-axis. Notice that this hyperbolic metric is symmetric around the $z$-axis, and hence there is an action by $\mathrm{SO}(2n)\times \mathbb R$, where the $\mathrm{SO}(2n)$ factor acts by rotations around the $z$-axis, whereas the $\mathbb R$ factor acts by translations in the $z$-direction. As a result, we can isometrically identify $C^{2n+1}_{\ell, \rho}$ with the quotient of $\mathbb R^{2n+1}$ by the element $(\rho, \ell)\in \mathrm{SO}(2n)\times \R$.

Now consider a new, rotationally symmetric metric on $\mathbb R^{2n+1}$, given by
	\begin{equation}\label{warping}
		h:=dr^2 + \sigma ^2(r) d\theta_{S^{2n-1}} ^2 + \tau ^2(r) dz ^2,
	\end{equation}
where $\sigma$ and $\tau$ are the functions constructed in Proposition \ref{interpolating-functions}. Note that by construction this metric still retains an isometric action of $\mathrm{SO}(2n)\times \mathbb R$, that is, the exact same action (setwise) is still an isometry for the $h$-metric. Thus we can quotient out by the same element $(\rho, \ell)\in \mathrm{SO}(2n)\times \mathbb R$ to get a new Riemannian metric on $C^{2n+1}_{\ell, \rho}$. 

To conclude, we are left with verifying that the metric we constructed satisfies properties (1)--(4). Recall that the functions $\sigma$ and $\tau$ constructed in Proposition \ref{interpolating-functions} satisfy
	$$\sigma (r) =
		\begin{cases}
			\sinh(r), & r\geq R \\
			r, & r\leq 1/R
			\end{cases} \hskip 0.5in 
	\tau (r) =
		\begin{cases}
			\cosh(r), & r\geq R \\
			1, & r\leq 1/R,
		\end{cases} $$
and thus that the metric $h$ satisfies
	$$h=
		\begin{cases}
			dr^2 + \sinh^2(r) d\theta_{S^{2n-1}} ^2  + \cosh^2(r) dz ^2, & r\geq R \\
			dr^2 + r^2 d\theta_{S^{2n-1}} ^2  + dz ^2, & r\leq 1/R.
		\end{cases}$$
Since these are the hyperbolic and Euclidean metric, respectively (in cylindrical coordinates), this immediately tells us that $h$ has properties (1) and (2).

We now verify property (3), that is, that $h$ has nonpositive sectional curvature. In the case where $n=1$ (i.e., on $\R^3$), the verification that $h$ has nonpositive sectional curvature was carried out in Corollary \ref{3d-NPC}, so we will henceforth assume that $n\geq2$. Now let $p\in \R^{2n+1}$ be an arbitrary point. We will check that the sectional curvature is nonpositive along all tangent $2$-planes at $p$. Since the metric is rotationally symmetric around the $z$-axis, we may as well assume $p=(r_0, 0, \ldots , 0, z_0)$ for some $r_0\geq 0$ and $z_0 \in \R$. Note that it is sufficient to establish nonpositive sectional curvature at all points where $r_0 >0$, that is, points that are not on the $z$-axis. Indeed, the curvature function is continuous, and we can approximate any tangent $2$-plane on the $z$-axis by a sequence of tangent $2$-planes off the $z$-axis. So we will henceforth assume that $r_0>0$. 

Let $H$ be an arbitrary tangent $2$-plane at the point $p$. The tangent space at $p$ decomposes as a direct sum of the $(2n-1)$-dimensional tangent space to the sphere through $p$, along with a pair of $1$-dimensional spaces spanned by $\frac{\partial}{\partial z}$ and $\frac{\partial}{\partial r}$, respectively. Observe that at the point $p$, the vectors $\{\frac{\partial}{\partial x_2}, \ldots ,\frac{\partial}{\partial x_{2n}}\}$ are vectors tangent to the $(2n-1)$-sphere $S^{2n-1}$ described by the equations $r=r_0, z=z_0$. Consider the projection of $H$ onto this subspace. It has dimension at most $2$. Since the rotations $\mathrm{SO}(2n)$ around the $z$-axis act transitively on orthonormal pairs of tangent vectors to $S^{2n-1}$, we can use an isometry to move $H$ (while fixing $p$) so that its projection now lies in the span of $\{\frac{\partial}{\partial x_2}, \frac{\partial}{\partial x_3}\}.$ Thus the sectional curvature along {\it any} tangent $2$-plane $H$ at $p$ coincides with the sectional curvature along a $2$-plane lying in the span of the four vectors $\{\frac{\partial}{\partial r}, \frac{\partial}{\partial x_2}, \frac{\partial}{\partial x_3}, \frac{\partial}{\partial z}\}$.

Now consider the subspace spanned by the coordinates $x_1, x_2, x_3, z$. This is a copy of $\R^4$ inside $\R^{2n+1}$, and 
by the discussion above, every tangent $2$-plane to $\R^{2n+1}$ can be moved by an isometry to lie inside this $\R^4$. Since 
the $\R^4$ is the fixed subset of an isometric $\mathrm{SO}(2n-3)$-action (recall $n\geq 2$), it is a totally geodesic subset with 
respect to the $h$-metric. So we are just left with computing the sectional curvature for $\R^4$ 
with the restricted metric.  But in Proposition \ref{4d-NPC}, we showed this metric on $\R^4$ has nonpositive sectional 
curvature, completing the verification of property (3).

Lastly, we need to check property (4). The space $C^{2n+1}_{l,\rho}$ is the quotient of $\mathbb R^{2n+1}$ by the isometry $(\rho, l) \in \mathrm{SO}(2n)\times \mathbb R$ (for the hyperbolic metric). Since the construction of the metric $h$ is invariant under the action of $\mathrm{SO}(2n)\times \mathbb R$, the metric $h$ descends to a metric on $C^{2n+1}_{l,\rho}$. Finally, parallel transport along the $z$-axis is given by (the differential of) the vertical translation. Since we are taking the quotient by $(\rho, l)$, it immediately follows that the holonomy along $\gamma$ is still given by the same matrix $\rho$. This verifies property (4) and completes the proof of the Proposition.
\end{proof}

%

\subsection{Completing the proof.}\label{subsec:thmB-proof}

We now have all the necessary ingredients.

\begin{proof}[Proof of Theorem \ref{thmA}]
Let $M$ be an arbitrary finite-volume odd-dimensional hyperbolic $(2n+1)$-manifold. Then by Proposition \ref{P:HolonomyProp}(a) we can find a closed geodesic $\eta$ in $M$ whose holonomy $\rho \in \mathrm{SO}(2n)$ is purely irrational. The geodesic $\eta$ might have self-intersections and might have small normal injectivity radius, but by passing to a finite cover $\bar M \rightarrow M$, we can ensure that there is an embedded lift $\bar \eta \hookrightarrow \bar M$
whose normal injectivity radius is $R>{37}$. By Lemma \ref{standard-nbhd}, the geodesic $\bar \eta$ has an $R$-neighborhood $N$ isometric
to an $R$-neighborhood of $\gamma \hookrightarrow C^{2n+1}_{l, \rho^k}$ for some integer $k\geq 1$. Of course, if $\rho$ is purely irrational, then so is $\rho^k$. Applying Proposition \ref{flattening-out}, we can flatten out the metric inside this neighborhood $N$. By construction, the metric satisfies properties (1) and (2) in the statement of our theorem.

To verify property (3), let us assume by way of contradiction that $\bar M$, equipped with the flattened out metric near $\gamma$, contains uncountably many closed geodesics. Since $\pi_1(\bar M)$ is countable, it follows that there exists a free homotopy class of loops which contains uncountably many geometrically distinct closed geodesics. Let $\eta_1, \eta_2$ be two such closed geodesics. From the flat strip theorem (see \cite[Prop. 5.1]{eberlein-oneill}), it follows that there exists an isometrically embedded flat cylinder that cobounds $\eta_1, \eta_2$ and hence that at each point along $\eta_i$, we have a tangent $2$-plane where the sectional curvature is zero. Now for our perturbed metric on $\bar M$, the sectional curvature is strictly negative {\it except} in the neighborhood $N$ of $\bar \eta$. Hence we see that the $\eta_i$ are entirely contained in $N$.

From the construction of the metric in $N$, the metric on $\tilde N$ is isometric to a neighborhood of the $z$-axis in the $h$ metric on $\R^{2n+1}$ (from the construction in Proposition \ref{flattening-out}). Under this identification, the lift of $\bar \eta$ is identified with the $z$-axis, whereas the lift $\tilde \eta_i$ of each $\eta_i$ is a geodesic whose $z$-coordinate is unbounded (in both directions). Since the lift of $\eta_i$ is at bounded distance from the $z$-axis, and since the $h$-metric on $\R^{2n+1}$ is nonpositively curved, $\tilde \eta_i$ cobounds a flat strip $S$ with the $z$-axis. We now claim $S$ has width zero, that is, that $\tilde \eta_i$ coincides with the $z$-axis. To see this, we note that $\pi_1(N) \cong \mathbb Z$ is generated by the loop $\bar \eta$. Under the identification with $\mathbb R^{2n+1}$, this element is represented by the element $g:=(\rho^k, \ell)\in \mathrm{SO}(2n)\times \R$. Since some power $g^s$ leaves $\tilde \eta_i$ invariant and acts by translation on the $z$-axis, it must leave the flat strip $S$ invariant. If $S$ has positive width, there exists a non-zero tangent vector $\vec v\in T_pS$ based at a point $p$ on the $z$-axis, which is orthogonal to the $z$-axis. Under the action of $g^s$, $\vec v$ is translated up along the flat strip $S$ to a tangent vector that is still orthogonal to the $z$-axis. Since $g^s = (\rho^{ks}, s\ell)$, this implies that $\vec v$ is an eigenvector for the matrix $\rho^{ks}$. But this contradicts the fact that $\rho^{ks}$ is purely irrational. We conclude no such vector exists and hence that $S$ must have width zero. This now forces each of the $\tilde \eta_i$ to coincide with the $z$-axis, and since they are freely homotopic, they have the same period, so must coincide geometrically, giving us the desired contradiction. This verifies property (3) and completes the proof of Theorem \ref{thmA}.
\end{proof}


\section{Closing theorem for fat flats}\label{sec:closing}

Before proving Theorem \ref{thmB} in full generality, we prove the $k=1$ case of the statement. This special case is easier to explain and contains all the key ideas for the general case. The $k=1$ case of Theorem \ref{thmB} can be rephrased as the following:

\begin{thm}\label{thm:closing}
If some geodesic $\gamma$ in $M$ {is a fat 1-flat}, then there exists a {\em closed} {fat} 1-flat in $M$. 
\end{thm}

After establishing preliminary lemmas in Section \ref{subsec:lemmas}, we prove Theorem \ref{thm:closing} in Section \ref{subsec:geodesic-case}. Theorem \ref{thm:closing} is easily seen to be true if $M$ is flat, and we will use this fact in our proof. We explain how to extend the argument to a proof of Theorem \ref{thmB} in Section \ref{subsec:further-comments}. 

This result can be interpreted as a closing lemma for fat flat neighborhoods (compare with the Anosov closing lemma, \cite[Cor. 18.1.8]{KH}). The argument is a version of that presented by Cao and Xavier in \cite{cao-xavier} for surfaces. The basic idea is the same, but a few complications arise in dimensions greater than two.

%
\subsection{Preliminaries}\label{subsec:lemmas}

Let $FM$ (resp., $F\tilde M$) denote the $n$-frame bundle over $M^n$ (resp., $\tilde M^n$). Frames are ordered orthonormal sets of $n$ vectors, and for a frame $\sigma$, $\sigma(i)$ denotes the $i^{th}$ member of $\sigma$.  Let $d_{FM}(\cdot,\cdot)$ denote the distance on $FM$ induced by the Riemannian metric. 

Let $Gr_kM$ (resp., $Gr_k\tilde M$) denote the $k$-Grassmann bundle over $M$ (resp., $\tilde M$), that is, the bundle whose fiber over $p$ is the Grassmannian of $k$-planes in the tangent space at $p$. Denote by $d_{Gr_kM}(-,-)$ or $d_{Gr_k\tilde M}(-,-)$ a metric on this space.

\begin{defn}\label{def:ff}
A \emph{framed {fat} flat} is a pair $(F,\sigma)$ where $F\subset \tilde M$ and $\sigma\in F\tilde M$ satisfy:
\begin{itemize}
	\item $F$ splits isometrically as $\mathbb{R}^k \times Y$ where $k\geq 1$, $Y\subset \mathbb{R}^{n-k}$ is compact, connected, and has non-empty interior.
	\item $\{\sigma(1),\ldots , \sigma(k)\}$ span the $\mathbb{R}^k$ factor in the splitting of $F$.
\end{itemize}
We will assume $Y$ is identified with a subset of $\mathbb{R}^{n-k}$ in such a way that it contains a neighborhood of 0 and that the basepoint of $\sigma$ has coordinate $0$ in the $Y$-factor. We refer to $Y$ as the \emph{cross-section} of $F$.
\end{defn}

Note that the sectional curvatures are identically zero on $F$. We use the terminology \emph{{fat flat}} to distinguish these from the usual notion of flats, which may be of dimension strictly smaller than $n$, and which play an important role for higher rank manifolds (see, e.g., \cite{bbe, bbs, ballmann-rank, burns-spatzier, Eberlein}).

We remark that our definition requires $k\leq n-1$, since $Y$ must have non-empty interior. We have also assumed $Y$ is compact. This is necessary for choosing a `maximal' framed {fat} flat (see Lemma \ref{lem:maximal}) and for a limiting argument (see Theorem \ref{thm:selection}).

\begin{defn}
For a framed {fat} flat $(F,\sigma)$, with $p$ the basepoint of $\sigma$, we denote the subspace of $T_p\tilde M$ spanned by the first $k$ vectors of $\sigma$ by $V^\infty(\sigma) \in Gr_k\tilde M$.
\end{defn}

%

We need the following Lemmas and other tools for the proof.

\begin{defn}
Fix $R>0$ and $\delta>0$ and let $C^l_{R,\delta}$ denote the set of compact, connected subsets of $\overline B_R(0)\subset \mathbb{R}^l$ which contain $B_\delta(0)$. We endow $C^l_{R,\delta}$ with the Hausdorff distance
\[ d_H(X,Y) = \inf\{ \epsilon: X\subseteq N_\epsilon(Y) \mbox{ and } Y \subseteq N_\epsilon(X)\}. \]
\end{defn}

With this distance, $C^l_{R,\delta}$ is a metric space.

\begin{thm}[Blaschke Selection Theorem; see, e.g., \cite{price}]\label{thm:selection}
$C^l_{R,\delta}$ is compact.
\end{thm}

\begin{proof}
Without the connectedness and $\delta$-ball restrictions, this is a direct application of the Selection Theorem since $\overline B_R(0)$ is compact. It is clear that we can add these restrictions as the Hausdorff-distance limit of compact, connected sets will be connected, and the Hausdorff-distance limit of sets containing $B_\delta(0)$ will contain $B_\delta(0)$.
\end{proof}

We also recall the following fact about nonpositive curvature. Recall that a \emph{flat strip} is a subset isometric to $\mathbb{R}\times [0,a]$ for some $a>0$.

\begin{thm}\cite[Prop. 5.1]{eberlein-oneill}
Let $\tilde\gamma_1(t)$ and $\tilde\gamma_2(t)$ be (set-wise) distinct geodesics in a simply-connected, nonpositively curved manifold $\tilde M$ such that $d(\tilde\gamma_1(t),\tilde\gamma_2(t))$ is bounded. Then there is a flat strip bounded by $\tilde\gamma_1$ and $\tilde\gamma_2$ in $\tilde M$.
\end{thm}

This theorem implies the following:

\begin{lem}\label{lem:convex}
Any framed {fat} flat in $\tilde M$ is contained in a framed {fat} flat for which the $Y$-factor is convex.
\end{lem}

\begin{proof}
Let $(F,\sigma)$ be a framed {fat} flat. For any distinct pair of points $y_1, y_2$ in $Y$ and any unit vector $v\in \mathbb{R}^k$, the geodesics $\tilde\gamma_i = \mathbb{R}v\times \{y_i\}$ are at bounded distance in $\tilde M$, and hence bound a flat strip. As this holds for all directions $v$ and all pairs $y_1, y_2$, the convex hull of $F$ in $\tilde M$ is flat. Together with $\sigma$, it is the desired framed {fat} flat.
\end{proof}

We also need the following simple geometric fact.

\begin{lem}\label{lem:translation}
Let $Y\subset \mathbb{R}^l$ be a compact, convex subset with non-empty interior. Let $A\in Isom(\mathbb{R}^l)$. If the translational part of $A$ is non-zero, then $vol_l(Y\cup A(Y)) > vol_l(Y)$.
\end{lem}

\begin{proof}
Write $A = B+v$ where $v$ is the translational part of $A$. Since $A$ is convex with nonempty interior, it is the closure of its interior, and so if $vol_l(Y\cup A(Y)) = vol_l(Y)$, then $A(Y)$ must be contained in $Y$.  Let $B$ be a ball of minimal radius containing $Y$. Then if $vol_l(Y\cup A(Y)) = vol_l(Y)$, then $A(Y)$ must belong to both $B$ and $B+v$. But if $v\neq 0$, then $A(Y)\subset B\cap (B+v)$ belongs to a ball of radius strictly smaller than that of $B$, contradicting the minimality of our choice of $B$.
\end{proof}

The main geometric tool behind our proof will be the following lemma.

\begin{lem}\label{lem:thicken}
In a nonositively curved, simply connected manifold $M^n$, let $(F_1,\sigma_1)$ and $(F_2,\sigma_2)$ be a pair of framed {fat} flats splitting isometrically as $\mathbb{R}^k\times Y_i$. Assume that $Y_i\in C^{n-k}_{R,\delta}$ for $i=1,2$. Then for all sufficiently small $\theta>0$, there exists $R(\theta)>0$ (which depends on the geometry of $F_i$ through $\delta$) satisfying $R(\theta)\to \infty$ as $\theta\to 0$ and such that:
\begin{itemize}
	\item If $d_{FM}(\sigma_1,\sigma_2)<\theta$, and $V^\infty_1\neq V^\infty_2$ as subspaces of $M^n$, then $F_1 \cup F_2$ contains a subset $X$ which splits isometrically as $B^{\mathbb{R}^k}_{R(\theta)}(0) \times Y'$ with its two factors lying along the respective factors of $F_1$, and $vol_{n-k}(Y')\geq  vol_{n-k}(Y_1)+C$ for some positive constant $C$, which depends only on $\delta$ and $n-k$.
\end{itemize}
\end{lem}

The idea of the proof is the following. When $d_{FM}(\sigma_1,\sigma_2)$ is very small but nonzero and $V^\infty_1\neq V^\infty_2$ as subspaces of $M^n$, $F_1$ and $F_2$ are close to being parallel, but not parallel. Thus, at some point, $F_1$ and $F_2$ diverge slightly. Some of this divergence must happen in the directions spanned by $Y_1$. Then $F_1$ and $F_2$ cover a slightly larger region $Y'$ in the directions of this factor and run nearly parallel, their union covering a large ball in the other $k$ directions.

Figure \ref{fig:overlap} illustrates the proof when $n=2$ and $k=1$.

\begin{figure}[h]
\begin{tikzpicture}

\path [fill=lightgray] (-5,0) -- (5,0) -- (5,2) -- (-5,2) -- (-5,0);
\path [fill=lightgray] (-5,-.2) -- (5,2.2) -- (5,4.2) -- (-5,1.8) -- (-5,-.2);
\path [fill=gray] (-5,.2) -- (5,2.6) -- (5,3.8) -- (-5,1.4) -- (-5,.2);

\draw (-5,0) --(5,0);
\draw (-5,2) --(5,2);
\node at (-2.5,-.3){$F_1$};

\draw[thick, ->](-4.5,1)--(-4,1);
\draw[thick, ->](-4.5,1)--(-4.5,1.5);
\node at (-4.2,1.25){$\sigma_1$};

\draw (-5,-.2) --(5,2.2);
\draw (-5,1.8) --(5,4.2);
\node at (-2.5,2.8){$F_2$};
\node at (5.3,3.1){$T_2$};

\draw[thick, ->](-4.7,.8)--(-4.2,.9);
\draw[thick, ->](-4.7,.8)--(-4.8,1.3);
\node at (-4.7,.6){$\sigma_2$};

\draw[thick, ->](-5.5,0)--(-5.5,2);
\node at (-5.7,1) {$Y_1$};

\draw[thick, ->](-1,-.5)--(1,-.5);
\node at (0,-1){$R_1\cong\mathbb{R}^k$};

\draw[ultra thick, pattern=north west lines](-.8,0.1)--(2.5,0.1)--(2.5,2.4)--(-.8,2.4)--(-.8,0.1);

\node at (1,2.7) {$X$};

\draw[thick, ->] (.9,1.2)--(1.4,1.2);
\draw[thick, ->] (.9,1.2)--(.9,1.7);
\node at (.8,1){$\sigma$};

\draw[thick, ->] (-1,0.2) -- (-1,2.2);
\node at (-1.3,1.3){$Y'$};

\end{tikzpicture}
\caption{Nearby framed flats cover a strictly thicker region $X$}\label{fig:overlap}
\end{figure}
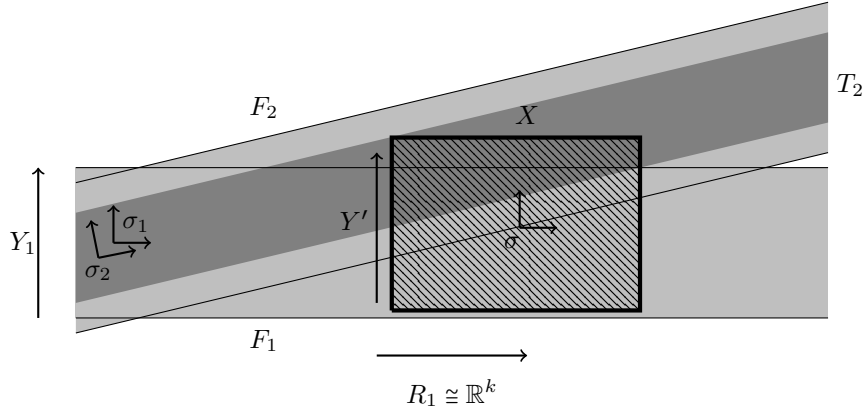

\begin{proof}
We give the proof under the assumption that $M^n = \mathbb{R}^n$ for simplicity. The entire argument deals with the interior of $F_1 \cup F_2$, where the space is flat because of the isometric splitting, so the argument extends to the case when the ambient manifold is $M^n$ unproblematically.

Fix the following notation. Write $F_i = R_i \times Y_i$. Note that each $Y_i$ factor lies along a subspace of $\mathbb{R}^n$, which we denote $W_i$. $F_2$ contains $T_2:=R_2\times B_2$ where $B_2 = B_\delta^{\mathbb{R}^{n-k}}(0)$ lies in the $Y_2$-factor. Let $S_2\subset T_2$ be $R_2\times B_2'$ where $B_2' = B_{\delta/2}^{\mathbb{R}^{n-k}}(0)$ in the $Y_2$-factor.

Fix $\theta$ smaller than $\delta$. In addition, consider the intersection of $S_2$ with the subspace $W_1$. We further choose $\theta$ so small that this intersection has $(n-k)$-volume at least half the volume of the $\delta/2$-ball. Let $C$ be $\frac{1}{4}$ the $(n-k)$-volume of the $\delta/2$ ball.

Since $\theta<\delta$ and the distance $d_{F\mathbb{R}^n}(\sigma_1,\sigma_2)$ is bounded below by the distance between the basepoints of these frames, $F_1\cap F_2\neq \emptyset$; in fact $F_1 \cap S_2 \neq \emptyset.$ Since $V^\infty_1 \neq V^\infty_2$, $(\{w\} \times B_2')\cap F_1 = \emptyset$ at some point $w\in R_2$; that is, near the basepoints of $\sigma_1$ and $\sigma_2$, $S_2$ intersects $F_1$, but at some point along the $\mathbb{R}^k$-factor of $F_2$, $S_2$ no longer intersects $F_1$. Therefore, at some coordinate $w'\in R_1$, the affine slices of $S_2$ parallel to $W_1$ intersect $F_1$ nontrivially, but in less than one-half of their $(n-k)$-volume. Hence, lying outside of $F_1$ we have a slice of $S_2\subset F_2$ in the $W_1$-direction of $(n-k)$-volume at least $\frac{1}{4}vol_{n-k}(B_{\delta/2}(0))$.

Let $Y'$ be the union of $Y_1$ with the $W_1$-slice of $S_2$ at the coordinate $w'\in R_1$. By the choice of $w'$ it has $(n-k)$-volume at least $vol_{n-k}(Y) + C$. Since a uniform neighborhood of $S_2$ is contained in $T_2\subset F_2$, for all $w''$ sufficiently close to $w'$ in $R_1$, the parallel translation of $Y'$ in the $R_1$-direction still belongs to $F_1\cup F_2$. Therefore $F_1 \cup F_2$ contains a subset splitting as $B^{\mathbb{R}^k}_R(0)\times Y'$ as desired.

Finally, we note that as $\theta \to 0$, $T_2, S_2,$ and $F_1$ become very nearly parallel in the $R_i$-directions. Therefore, the set of $w''$ for which the $Y'$ defined above still lies in $T_2 \subset F_2$ contains a larger and larger ball about $w'$ in the $R_1$-factor. A precise estimate of $R(\theta)$ could be given using some simple Euclidean geometry, but this argument is sufficient to establish that $R$ can be taken to depend only on $\theta$ and $\delta$, with $R(\theta)\to \infty$ as $\theta \to 0$. This finishes the proof.
\end{proof}

\begin{rem}
The assumption $V^\infty_1\neq V^\infty_2$ as subspaces of $M^n$ in Lemma \ref{lem:thicken} is necessary to ensure we can take $C$ independent of $\theta$. If $V^\infty_1 = V^\infty_2$ as subspaces of $M^n$ and $d_{FM}(\sigma_1,\sigma_2)$ is small and non-zero, then $F_1$ and $F_2$ are nearby and parallel in their infinite directions. Then it is fairly clear that $F_1\cup F_2$ form a framed {fat} flat with strictly larger cross-section (although in this case the increase in volume $C$ goes to 0 as $\theta$ does). We will use this idea (and prove this statement precisely) in our proof of Theorem \ref{thm:closing}. We separate the argument into two cases according to whether $V^\infty_1$ and $V^\infty_2$ are parallel because the non-parallel case involves a limiting argument in which the uniform choice of $C$ will be useful.
\end{rem}

We supply a definition for the objects built by Lemma \ref{lem:thicken}.

\begin{defn} \label{flatbox}
We call $(X,\sigma)$ a \emph{framed flat box} if $X\subset \tilde M$ splits isometrically as $B^{\mathbb{R}^k}_R(0) \times Y'$ with $Y'$ a compact, connected subset of $\mathbb{R}^{n-k}$ with nonempty interior, and $\sigma$ is a frame based at a point in $X$ whose first $k$ vectors are tangent to the $B_R(0)$ factor. We call $Y$ the \emph{cross section} of $(X,\sigma)$ and $R$ its \emph{length}.
\end{defn}

Equipping the $X$ built by Lemma \ref{lem:thicken} with a frame parallel to $\sigma_1$ and with basepoint having coordinate 0 in the $B_R(0)$-factor of $X$, we see that Lemma \ref{lem:thicken} supplies a framed flat box with cross section of strictly larger $(n-k)$-volume than that of the original {fat} flats.

%

\subsection{Proof of Theorem \ref{thm:closing}}\label{subsec:geodesic-case}

The idea of the proof is as follows. If there is no closed {fat} 1-flat in $M$, then the {fat} 1-flats must have accumulation points. When equipped with a frame $\sigma$ with $\sigma(1)=\dot\gamma(0)$, a {fat} flat becomes a framed {fat} flat. Pick a framed {fat} flat whose cross-section $Y$ has maximal $(n-k)$-volume. As this {fat} flat accumulates on itself, using Lemma \ref{lem:thicken}, we will build in $\tilde M$ framed flat boxes with strictly larger cross-sections and increasingly large lengths. Using the compactness of $M$, we find a sequence $(X_n,\sigma_n)$, of framed flat boxes in $\tilde M$ such that $\sigma_n \to \sigma^*$, the length of $X_n$ goes to infinity, and -- using the Selection Theorem (Theorem \ref{thm:selection}) -- $Y_n$ converges to $Y^*$. The result is a framed {fat} flat with strictly larger cross-sectional $(n-k)$-volume, contradicting the choice of the original framed {fat} flat. Therefore our original framed {fat} flat of maximal cross-sectional area should contain a closed geodesic, as desired.

This is the scheme of proof used by Cao and Xavier. In higher dimensions, there are two extra hurdles to overcome. First, we need to show that, unless $M$ is flat, framed {fat} flats of maximal cross-sectional volume exist. This is carried out in Lemma \ref{lem:maximal}; in dimension two, this is easily dealt with by choosing the flat strip of greatest width. Second, we need to consider the case where $V^\infty_1=V^\infty_2$ when the framed {fat} flat accumulates on itself, ruling out an application of Lemma \ref{lem:thicken} and breaking the described scheme of proof. We deal with the geometric situation that arises in this case separately in Lemma \ref{lem:submfld}. We note again that, in dimension two, the geometry of this situation is particularly simple. In addition to these two hurdles, we must also take care with our limiting procedure and use the selection theorem, an issue that does not arise in dimension two.

\begin{lem}\label{lem:maximal}
The framed {fat} flats in $\tilde M$ satisfy the following two properties:
\begin{itemize}
	\item[(a)] There exists some maximal $k\leq n$ such that there is a framed {fat} flat $(F,\sigma)$ with $F\cong \mathbb{R}^k\times Y$.
	\item[(b)] For this value of $k$, choose $\delta>0$ so that there is a framed {fat} flat $F\cong \mathbb{R}^k \times Y$ with $B^{\mathbb{R}^{n-k}}_\delta(0)\subset Y$. Unless $M$ is flat, among those framed {fat} flats with noncompact factor of this maximal dimension whose cross-sections contain $B^{\mathbb{R}^{n-k}}_\delta(0)$, there is one with $vol_{n-k}(Y)$ maximal.
\end{itemize}
\end{lem}

\begin{proof}
\emph{(a)} This is trivial. Let $\mathcal{F}$ be the set of all framed {fat} flats in $\tilde M$. Each is isometric to $\mathbb{R}^k\times Y$ for some $k\geq 1$, so there is a maximal $k$ which appears. Choose $\delta$ as described in \emph{(b)} and let $\mathcal{F}^*$ consist of the framed {fat} flats achieving the maximal value of $k$ and with cross-sections containing $B^{\mathbb{R}^{n-k}}_\delta(0)$.

\emph{(b)} First, we claim that $Y$ is uniformly bounded over all $(F,\sigma)$ in $\mathcal{F}^*$.

Suppose, toward a contradiction, that $(F_n,\sigma_n) \in \mathcal{F}^*$ is a sequence of framed {fat} flats such that $Y_n$ contains a point $y_n$ with $|y_n|>n$ (as a point in $\mathbb{R}^{n-k}$). By Lemma \ref{lem:convex} we may assume $Y_n$ is convex. Let $v_n=y_n/|y_n|$ be the unit vector from 0 towards $y_n$. We can consider $v_n$ as a unit tangent vector in $\tilde M$ based at the basepoint of $\sigma_n$. Consider the sequence $(\pi_*\sigma_n, \pi_*v_n) \in FM \times T^1(M).$  Passing to a subsequence, it has a limit point, since $M$ is compact. Therefore, there is a sequence $(\gamma_{n_i})$ in $\Gamma$ such that $D\gamma_{n_i} (\sigma_{n_i},v_{n_i})$ accumulates to $(\sigma^*, v^*) \in F\tilde M \times T^1\tilde M$.

Consider the sequence of convex framed {fat} flats $(\gamma_{n_i} \cdot F_{n_i}, D\gamma_{n_i}\sigma_{n_i})$ in $\mathcal{F}^*$. Restrict each $Y_{n_i}$ to the convex hull in $\mathbb{R}^{n-k}$ of $B_\delta(0)$ and $y_n$. Then take the limit of $(F_{n_i},\sigma_{n_i})$, using Hausdorff convergence in the first factor. This limit exists by the following argument. By construction, $\sigma_{n_i} \to \sigma^*$, and so the $\delta$-balls around $0$ in $Y_{n_i}$ converge to $B^*$, the $\delta$-ball around 0 in the flat subspace spanned by the final $(n-k)$ vectors of $\sigma^*$. Since $D\gamma_{n_i}v_{n_i} \to v^*$ and all the $Y_{n_i}$ are convex, the limit $Y^*$ of the $Y_{n_i}$ is then equal to $\bigcup_{\lambda>0} (B^*+\lambda v^*)$.

The limit $(F^*,\sigma^*)$ is similar to a framed {fat} flat, except that $Y^*$ is non-compact, containing the half-infinite ray in the $v^*$ direction. Consider the sequence $(F^*, \sigma^*_n)$ where $\sigma^*_n = P_{nv^*}\sigma^*$ and $P_{nv^*}$ denotes parallel translation along the vector $nv^* \in Y^* \subset F^*$. Note that $(F^*, \sigma^*_n)$ contains a subset splitting isometrically as $\mathbb{R}^k \times [-n,n]\times \overline B^{\mathbb{R}^{n-k-1}}_\delta(0)$. The first factor is the $\mathbb{R}^k$ factor from $F^*$, the second factor lies along the $v^*$ direction, and the third in the remaining directions. Replacing the frames $P_{nv^*}\sigma^*$ with frames which have the same first $k$ vectors and $v^*$ as their $k+1^{st}$, these are framed {fat} flats, which we denote $(F'_n,\sigma'_n)$.

As before, project the $(F'_n,\sigma'_n)$ down to $M$, find a limit point of the $\pi_*\sigma'_n$, and then lift back to $\tilde M$. We obtain a sequence $\gamma_j$ such that $(\gamma_j \cdot F'_{n_j}, D\gamma_j \sigma'_{n_j})$ converges in $\tilde M$ (again using Hausdorff convergence in the first term). The limit extends infinitely in both directions parallel to $\lim_{j}D\gamma_jv^*_{n_j}$ and contains the $\delta$-ball around 0 in the final $n-k-1$ directions. Therefore it is a framed {fat} flat with noncompact factor of dimension $k+1$, strictly higher than our maximal dimension.

This is a contradiction to our choice of $k$ unless $k+1=n$. If $k+1=n$, $\tilde M$ is flat. For all other values of $k$, the contradiction implies that the convex framed {fat} flats in $\mathcal{F^*}$ with cross-section containing a $\delta$-ball have uniformly bounded $Y$-factors and hence uniformly bounded $vol_{n-k}(Y)$. We can then pick one with $vol_{n-k}(Y)$ maximal.
\end{proof}

We call a framed {fat} flat satisfying the requirements of Lemma \ref{lem:maximal} a \emph{maximal framed {fat} flat}. Before proving the next lemma, we remark that when two {fat} flats $F_1$ and $F_2$ have nonempty intersection, we can determine whether two subspaces $V_1\subset T_{p_1}\tilde M$ and $V_2\subset T_{p_2}\tilde M$ with basepoints in these sets are parallel by comparing them by parallel translation through $F_1 \cup F_2$.

\begin{lem}\label{lem:submfld}
Let $\tilde N$ be a complete $k$-dimensional submanifold of $\tilde M$. Suppose that there exists $\epsilon_0>0$ such that whenever $d_{Gr_k\tilde M}(T_p\tilde N, D\gamma(T_q\tilde N))<\epsilon_0$ for any $p,q\in\tilde N$, $\gamma\cdot\tilde N = \tilde N$. Then $\pi(\tilde N)$ is a complete, closed, and immersed submanifold of $M$.
\end{lem}

\begin{proof}
The covering map $\pi$, when restricted to $\tilde N$, is an immersion. Since $M$ is complete, to prove the lemma, we only need to verify that $N:=\pi(\tilde N)$ is closed.

Let $(p_n)$ be a sequence in $N$ converging to a limit $p^*$ in $M$. We want to show that $p^*\in N$. Fix a lift $\tilde p^*$ of $p^*$ in $\tilde M$ and a sequence of lifts $(\tilde p_n)$ with $\tilde p_n \in \gamma_n\tilde N$ and $\tilde p_n \to \tilde p^*$. Let $V_n = T_{\tilde p_n}(\gamma_n\tilde N) \in Gr_k\tilde M$. Since the fibers of $Gr_k\tilde M$ are compact, there exists a subsequence $(n_i)$ such that $V_{n_i} \to V^*$ in $Gr_k(\tilde M)$.  Then for all sufficiently large $i$ and $j$, we have $d_{Gr_k\tilde M}(V_{n_i}, V_{n_j})<\epsilon_0$. Then, by the assumption of the lemma, there is a fixed $\gamma_0$ such that for all sufficiently large $i$, $V_{n_i}$ is tangent to $\gamma_0\tilde N$. Therefore, the lifts $\tilde p_{n_i}$ lie in $\gamma_0
 \tilde N \cap \overline B_{\epsilon_0}(\tilde p^*)$. Since this set is closed, the limit of the $\tilde p_{n_i}$ belongs to it and, in particular, belongs to $\gamma_0\tilde N$. Hence $p^*\in N$.
\end{proof}
We are now ready to prove Theorem \ref{thm:closing}.

\begin{proof}[Proof of Theorem \ref{thm:closing}]

As noted before, the lift of any {fat} 1-flat is a framed {fat} flat (after equipping it with a frame). Unless $M$ is flat, we can choose a maximal one using Lemma \ref{lem:maximal}.  If $M$ is flat, the theorem is, of course, trivial.

Let $(F^*, \sigma^*)$ be a maximal framed {fat} flat, with cross-section $Y^*$ and $\sigma^*$ based at $\tilde p_0$. Let $\delta>0$ be the value appearing in the statement of Lemma \ref{lem:maximal}. That is, $(F^*,\sigma^*)$ has maximal dimension of its noncompact factor and maximal cross-sectional area among those framed {fat} flats whose cross-section contains a $\delta$-ball. Let $\tilde N = \exp_{\tilde p_0}(V^\infty(\sigma^*))$. This is a complete, totally geodesic submanifold of $\tilde M$ since $\tilde M$ is nonpositively curved and $\tilde N$ is flat.

We will initially give the proof under an assumption, which we will then verify. Recall that since $\tilde N$ lies inside a {fat} flat, we can speak of parallel elements in $Gr_k\tilde M$ when their basepoints are $\delta$-close to $\tilde N$.

\

\noindent\textbf{\textsc{Assumption (A).}} \emph{There exists $\epsilon_0>0$  and $<\delta$ such that whenever, for some $p,q\in\tilde N$ and $\gamma\in\Gamma$, we have $d_{Gr_k\tilde M}(T_p\tilde N,D\gamma(T_q\tilde N)) < \epsilon_0$, $T_p\tilde N$ and $T_q\tilde N$ are parallel.}

\

Under Assumption (A), a short argument using the maximality of $(F^*,\sigma^*)$ allows us to apply Lemma \ref{lem:submfld}.  Since $\epsilon_0<\delta$, we have that whenever $T_p\tilde N$ and $D\gamma(T_q\tilde N)$ are less than $\epsilon_0$ apart, $F^*$ and $\gamma \cdot F^*$ intersect. Since the tangent spaces to their $\mathbb{R}^k$-factors are parallel, $(F^*\cup \gamma\cdot F^*, \sigma^*)$ is a framed {fat} flat. Its cross-section is $Y^*\cup A(Y^*)$, where $A \in \mathrm{Isom}(\mathbb{R}^{n-k})$ is induced by the isometry $\gamma$.

Note that the translational part of $A$ will be zero if and only if $D\gamma(T_q\tilde N) = T_p\tilde N$. By Lemma \ref{lem:translation}, if the translational part of $A$ is non-zero, then $Y^* \cup A(Y^*)$ has strictly larger $(n-k)$-volume than that of $Y^*$. But this contradicts the maximality of $(F^*,\sigma^*)$. Therefore we conclude that $D\gamma(T_q\tilde N) = T_p\tilde N$, and we can invoke Lemma \ref{lem:submfld} to conclude that $N=\pi(\tilde N)$ is a complete, closed, immersed submanifold of $M$. Since the submanifold $\tilde N$ is totally geodesic and flat, so is its projection $N$. Therefore, it contains closed geodesics, any of which is {fat}, finishing the proof.

To complete the proof, it remains to verify that our Assumption (A) holds. Suppose not. Then the following must hold: \emph{There exist sequences $p_n, q_n \in \tilde N$ and a sequence $\gamma_n$ in $\Gamma$ satisfying $d_{Gr_k\tilde M}(T_{p_n}\tilde N, D\gamma_n(T_{q_n}\tilde N))\to 0$ with $T_{p_n}\tilde N$ and $D\gamma_n(T_{q_n}\tilde N)$ not parallel.}

By projecting down to the compact space $Gr_k(M)$, passing to a subsequence, and then lifting back to $Gr_k(\tilde M)$, we can take $p_n=p$ to be constant. Choose a frame $\hat \sigma$ based at $p$ whose first $k$ vectors span $T_p\tilde N$, and choose frames $\sigma_n$ at $q_n$ such that $(\gamma_n \cdot F^*, D\gamma_n\sigma_n)$ are framed {fat} flats and $d_{F\tilde M}(\hat\sigma, D\gamma_n\sigma_n)\to 0$. We can do this because the only restriction on choosing frames is that the first $k$ vectors span the infinite directions of the framed flat, and we know by assumption that these subspaces approach each other. Then $(F^*, \hat \sigma)$ and $(\gamma_n \cdot F^*, D\gamma_n\sigma_n)$ are framed {fat} flats with $Y$-factors in $C^{n-k}_{R,\delta}$ for some $R>0$ and $d_{F\tilde M}(\hat\sigma, D\gamma_n\sigma_n)\to 0$.

By assumption $V^\infty(\hat\sigma) \neq V^\infty(\sigma_n)$, so we can apply Lemma \ref{lem:thicken} and construct in $\tilde M$ a sequence $(X_n,\sigma'_n)$ of framed flat boxes with cross sections $Y_n \in C^{n-k}_{2R,\delta}$ of volume at least $vol_{n-k}(Y^*)+C$ and with lengths going to $\infty$. Project the frames $\sigma'_n$ to the compact $FM$ and pass to a convergent subsequence $\pi_*\sigma'_n\to \pi_*\sigma'$. Using the selection theorem (Theorem \ref{thm:selection}) in $C^{n-k}_{2R,\delta}$, pass to a further subsequence so that $Y_n\to Y'$, where we identify $Y_n$ with a subset of $\mathbb{R}^{n-k}$ using the final $n-k$ vectors of $\sigma'_n$. Note that $vol_{n-k}(Y')$ is strictly greater than that of $Y^*$. Take lifts of this subsequence of frames in $FM$ to frames in $F\tilde M$ which converge to $\sigma'$. We then see that $\sigma'$ lies in a framed {fat} flat with cross-section $Y'$, contradicting the maximality of $(F^*,\sigma^*)$. This contradiction shows that this situation does not arise for maximal framed {fat} flats, and thus Assumption (A) must hold. This finishes the proof.
\end{proof}

%

\subsection{The case $k\geq2$}\label{subsec:further-comments}

The proof of Theorem \ref{thm:closing} demonstrates Theorem \ref{thmB}, which we now restate. 

\begin{thm}\label{thm:kclose}
Suppose that $\tilde M$ contains a {fat} $k$-flat $F$. Then $M$ contains an immersed, closed, totally geodesic {fat} $k$-flat. 
\end{thm}

\begin{proof}
We can equip $N_w(F)$ with a frame to be a framed {fat} $k$-flat. Since framed {fat} flats exist, there are maximal framed {fat} flats as described by Lemma \ref{lem:maximal}. Let $(F^*,\sigma^*)$ be such a maximal framed {fat} flat. Its noncompact factor has dimension $k^* \geq k$. Let $\tilde N = \exp(\sigma^*)$. The proof of Theorem \ref{thm:closing} shows that $N:=\pi(\tilde N)$ is an immersed, closed, and totally geodesic $k^*$-dimensional submanifold with a flat neighborhood. Then for any dimension less than or equal to $k^*$, in particular $k$, we can take a closed, totally geodesic $k$-submanifold of $N$ as the immersed, totally geodesic {fat} $k$-flat required for the theorem.
\end{proof}

As a consequence, we see that if the unbounded region of zero curvature in $\tilde M$ contains a flat of dimension $k\geq 2$,
then the manifold $M$ contains uncountably many closed geodesics. More generally, we have the following:

\begin{cor}
Suppose that $\tilde M$ contains a {fat} $k$-flat. Then for all $1\leq l<k$, $M$ contains uncountably many immersed, closed, and totally geodesic, flat $l$-submanifolds. 
\end{cor}

\begin{proof}
Using Theorem \ref{thm:kclose}, there is an immersed, closed, and totally geodesic, flat $k$-submanifold $N$ in $M$. For each $1\leq l<k$, such a manifold has uncountably many closed, totally geodesic $l$-submanifolds. These are the submanifolds we want.
\end{proof}

\section{Dynamics of the geodesic flow.} \label{dynamics}

We turn our attention to dynamical properties of the geodesic flow for rank $1$ manifolds where the higher rank geodesics are precisely those coming from the zero curvature neighborhoods of {fat} flats, particularly, the examples provided by Theorem \ref{thmA}. The presence of the {fat} flat rules out the possibility of applying many of the powerful techniques of hyperbolic dynamics (e.g., establishing conjugacy with a suspension flow over a shift of finite type or establishing the specification property \cite{CLT}). Despite these difficulties, recent work by Burns, Climenhaga, Fisher, and Thompson \cite{BCFT} has developed the theory of equilibrium states for rank $1$ geodesic flows.  We give a version of the results proved there, adapted to the particular case of the examples introduced in this paper. We are able to obtain stronger conclusions than in the general case due to the explicit and simple characterization of the singular set. Definitions of the terms that appear in the following theorem are given in \cite{BCFT}.

\begin{thm}
Let  $(\bar M,g)$ be a manifold given by Theorem A. Let $h$ be the topological entropy of the geodesic flow. We have the following properties:
\begin{enumerate}
\item If $\varphi$ is H\"older continuous, and $\sup_{x\in \mathrm{Sing}}\varphi(x)-\inf_{x\in T^1M}\varphi(x) < h$, then $\varphi$ has a unique equilibrium state. This measure is fully supported;
\item Let $\varphi^u$ be the geometric potential. The potentials $q \varphi^u$ have a unique equilibrium state for all $q<1$. This measure is fully supported;
\item The Liouville measure is ergodic, and is an equilibrium state for $\varphi^u$. Any measure supported on the singular set is also an equilibrium state for $\varphi^u$;
\item Weighted closed geodesics equidistribute to the unique equilibrium measures in (1) and (2).
\end{enumerate}
\end{thm}

Property (1) is true for any rank $1$ manifold whose singular set has zero entropy. If, additionally, $\varphi^u$ vanishes on the singular set, then property (2) holds. This is always the case for a manifold which has all sectional curvatures negative away from the zero-curvature neighborhoods of some fat flats. Property (4) requires the cylindrical neighborhood to be twisted, or else there are uncountably many singular geodesics. This is the main additional property that one gains in the twisted cylinder case.

\begin{proof}
Properties (1), (2), and (4) are obtained from the main theorems of \cite{BCFT}. Write $\Sing$ for the singular set,  $h(\Sing)$ for the topological entropy of the flow restricted to the singular set, $P(\varphi)$ for the topological pressure and $P(\Sing, \varphi)$ for the topological pressure for the singular set. 

For (1), it is shown in \cite{BCFT} that, for a closed rank $1$ manifold $M$ and a H\"older continuous potential function $\varphi$, if $\sup_{x\in \Sing}\varphi(x)-\inf_{x\in T^1M}\varphi(x) < h - h(\Sing)$, then $\varphi$ has a unique equilibrium state, and this equilibrium state is fully supported. The set $\Sing$ can be characterized as those $x \in T^1M$ such that in the universal cover, the geodesic passing through $\tilde x$ is contained in the zero-curvature neighborhood of the fat flat. This set clearly has zero entropy since the flat geometry shows that a uniformly bounded number of geodesics is sufficient to $(t, \epsilon)$-span the singular set for any $t$.

For (2), we use the result from \cite{BCFT} that for $q \varphi^u$ to have a unique equilibrium state, it suffices to show that $P(\Sing, q \varphi^u)< P(q \varphi^u)$.  We know that $h(\Sing)=0$ and $\varphi^u(x)=0$ for all $x\in \Sing$. Hence, $P(\Sing, q \varphi^u)=0$ for all $q$.  
We show that $P(q \varphi^u)>0$ for all $q\in (-\infty, 1)$. Let $\mu_L$ denote the Liouville measure. It follows from the Pesin entropy formula that
\[
h_{\mu_L} + \int q \varphi^u d\mu_L = h_{\mu_L} - q \lambda^+(\mu_L)  >  h_{\mu_L} -  \lambda^+(\mu_L)  = 0.
\]
By the variational principle, $P(q\varphi^u)\geq h_{\mu_L} - \int q \varphi^u d\mu_L > 0$. Thus, $P(\Sing, q\varphi^u)< P(q \varphi^u)$ for all $q\in (-\infty, 1)$. 

For (3), it is well known that the Pesin entropy formula implies that $\mu_L$ is an equilibrium state for $\varphi^u$ and that the restriction of $\mu_L$ to the regular set is ergodic. To show that $\mu_L$ is ergodic, we only require that $\mu_L(\Sing)=0$. This is clear from the characterization of $\Sing$ given above. Finally, let $\mu$ be a measure supported on $\Sing$ (e.g., the measure corresponding to the closed geodesic $\gamma$). By the variational principle, $h_{\mu} \leq h(\Sing)=0$, and $\varphi^u$ vanishes on $\Sing$, so $h_{\mu} + \int \varphi^u d\mu  = 0$, and by the Margulis--Ruelle inequality and Pesin formula,  $P(\varphi^u) =0$. Thus $\mu$ is an equilibrium state for $\varphi^u$.

For (4), it is shown in \cite{BCFT} that weighted \emph{regular} closed geodesics equidistribute to the unique equilibrium measures in (1) and (2). Since there is only one closed geodesic that is not regular, this geodesic contributes no mass in the limit, and so in this situation, we have that the collection of \emph{all} weighted closed geodesics equidistributes to the unique equilibrium measures in (1) and (2).
\end{proof}

\noindent {\bf Acknowledgments:} The authors would like to thank the anonymous referee for several helpful suggestions.

\bibliographystyle{alpha}
\bibliography{biblio}

\end{document}